\documentclass[a4paper, 11pt]{article}

\usepackage[utf8]{inputenc}
\usepackage[T1]{fontenc}
\usepackage{graphicx}
\usepackage{xcolor}
\usepackage[english]{babel}
\usepackage{hyperref}
\usepackage{fancyhdr}
\usepackage{amsmath, amstext, amsopn, amssymb, amsthm}
\usepackage{thmtools}
\usepackage{mathtools}
\usepackage{bm}
\usepackage{bbm}
\usepackage[shortlabels, inline]{enumitem}
\usepackage[left=3cm,right=3cm,top=3cm,bottom=3cm]{geometry}
\usepackage{microtype}

\addto\extrasenglish{

}
\hypersetup{
	colorlinks   = true,
	urlcolor     = blue,
	linkcolor    = blue,
	citecolor    = blue
}

\DeclareMathOperator{\opint}{int}
\DeclareMathOperator{\supp}{supp}

\DeclareMathOperator{\W}{W}
\DeclareMathOperator{\id}{Id}

\DeclarePairedDelimiter{\norm}{\lVert}{\rVert}
\DeclarePairedDelimiter{\abs}{\lvert}{\rvert}
\DeclarePairedDelimiterX{\inner}[2]{\langle}{\rangle}{{#1},{#2}}

\newcommand{\R}{\mathbb{R}}
\newcommand{\N}{\mathbb{N}}

\newcommand{\OT}{\mathcal{T}}
\newcommand{\calV}{\mathcal{V}}
\newcommand{\calX}{\mathcal{X}}
\newcommand{\calY}{\mathcal{Y}}
\newcommand{\calP}{\mathcal{P}}
\newcommand{\calS}{\mathcal{S}}
\newcommand{\calU}{\mathcal{U}}

\newcommand{\calF}{\mathcal{F}}
\newcommand{\calC}{\mathcal{C}}
\newcommand{\calH}{\mathcal{H}}
\newcommand{\calL}{\mathcal{L}}
\newcommand{\calI}{\mathcal{I}}
\newcommand{\calN}{\mathcal{N}}

\makeatletter
\newcommand*{\bigcdot}{}%
\DeclareRobustCommand*{\bigcdot}{%
	\mathbin{\mathpalette\bigcdot@{}}%
}
\newcommand*{\bigcdot@scalefactor}{.5}
\newcommand*{\bigcdot@widthfactor}{1.15}
\newcommand*{\bigcdot@}[2]{%
	\sbox0{$#1\vcenter{}$}%
	\sbox2{$#1\cdot\m@th$}%
	\hbox to \bigcdot@widthfactor\wd2{%
		\hfil
		\raise\ht0\hbox{%
			\scalebox{\bigcdot@scalefactor}{%
				\lower\ht0\hbox{$#1\bullet\m@th$}%
			}%
		}%
		\hfil
	}%
}
\makeatother
\newcommand{\argdot}{\,\bigcdot\,}
\newcommand{\restr}[2]{{\left.\kern-\nulldelimiterspace #1 \vphantom{\big|} \right|_{#2}}}
\newcommand{\indic}{\mathop{}\!\mathbb{I}}
\newcommand{\indicfunc}[1]{\indic({#1})}
\newcommand{\de}{\mathop{}\!\mathrm{d}}
\makeatletter
\newcommand*{\transp}{%
	{\mathpalette\@transpose{}}%
}
\newcommand*{\@transpose}[2]{%
	\raisebox{\depth}{$\m@th#1\intercal$}%
}
\makeatother
\newcommand{\defeq}{\coloneqq}

\theoremstyle{plain}
\newtheorem{sectioncount}{xxxxxxx}[section]
\newtheorem{theorem}[sectioncount]{Theorem}

\newtheorem{lemma}[sectioncount]{Lemma}
\newtheorem{corollary}[sectioncount]{Corollary}
\theoremstyle{definition}
\newtheorem{example}[sectioncount]{Example}

\newtheorem{remark}[sectioncount]{Remark}

\usepackage[backend = biber, maxbibnames = 10, sortcites = true, giveninits = true]{biblatex}
\usepackage{csquotes}
\newcommand{\footremember}[2]{
	\footnote{#2}
	\newcounter{#1}
	\setcounter{#1}{\value{footnote}}
}
\newcommand{\footrecall}[1]{
	\footnotemark[\value{#1}]
}

\begin{document}

\title{Nonlinear Inverse Optimal Transport: Identifiability of the Transport Cost from its Marginals and Optimal Values\thanks{
\textbf{Founding: }The work of the second author was supported by the Deutsche Forschungsgemeinschaft (DFG, German Research Foundation) through the GRK 2088 and the work of the third author by the DFG through the SFB 1456.}}

\author{Alberto Gonz\'{a}lez-Sanz\footremember{columbia}{Department of Statistics, Columbia University, New York, United States (\href{mailto:ag4855@columbia.edu}{ag4855@columbia.edu}).}
\and Michel Groppe\footremember{ims}{Institute for Mathematical Stochastics, University of Göttingen, 37077 Göttingen, Germany (\href{mailto:michel.groppe@uni-goettingen.de}{michel.groppe@uni-goettingen.de}, \href{mailto:munk@math.uni-goettingen.de}{munk@math.uni-goettingen.de}).}
\and Axel Munk\footrecall{ims}}
\date{June 18, 2024}

\maketitle

\begin{abstract}
The inverse optimal transport problem is to find the underlying cost function from the knowledge of optimal transport plans. While this amounts to solving a linear inverse problem, in this work we will be concerned with the nonlinear inverse problem to identify the cost function when only a set of marginals and its corresponding optimal values are given. We focus on absolutely continuous probability distributions with respect to the $d$-dimensional Lebesgue measure and classes of concave and convex cost functions. Our main result implies that the cost function is uniquely determined from the union of the ranges of the gradients of the optimal potentials. Since, in general, the optimal potentials may not be observed, we derive sufficient conditions for their identifiability --- if an open set of marginals is observed, the optimal potentials are then identified via the value of the optimal costs. We conclude with a more in-depth study of this problem in the univariate case, where an explicit representation of the transport plan is available. Here, we link the notion of identifiability of the cost function with that of statistical completeness.
\end{abstract}

\noindent \textit{Keywords}: inverse optimal transport, identifiability, cost function, completeness, optimal potentials

\noindent \textit{MSC 2020 subject classification}: 49Q22, 45Q05, 60E10

\section{Introduction} \label{sec:intro}

Based on the theory of optimal transport (OT), a long standing and well investigated topic in mathematics and related fields (see e.g.\ the seminal paper by Kantorovich \& Rubinstein \cite{Kantorovich58} or \cite{Rachev98a,Rachev98b,villani2008optimal,Ambrosio2005} for textbook references), real world data analysis based on OT (and related notions) has become a widespread area of research only more recently, mainly due to significant computational progress (see e.g.\ \cite{Peyr2018ComputationalOT,Santambrogio2015}) and the development of a profound statistical theory (see e.g.\ \cite{Panaretos2020}). Beyond economics (see e.g.\ \cite{Kantorovich1960,GalichonBook}) applications meanwhile range from mathematical imaging \cite{Bonneel2011,Courty2014OptimalTF,Lara} and computer graphics \cite{Bonneel23} to machine learning (fair learning \cite{pmlr-v97-gordaliza19a,Risser2019TacklingAB}, generative models \cite{Arjovsky2017WassersteinG,Pooladian2023MultisampleFM}, manifold learning \cite{zhang2023manifold}) to statistics (multivariate quantiles \cite{ChernozhukovQuantiles,Hallin}, dependency modeling \cite{nies2023transport,Wiesel2022} and goodness of fit testing \cite{Deb2023,del1999tests}, inference \cite{Sommerfeld2017,Panaretos2019,Munk1998} and regression \cite{carlierQuantile,delbarrio2022nonparametric,Chen2023}) or domain applications such as cell biology \cite{tameling2021colocalization,Schiebinger2019}, to mention a few.

The OT problem can be informally defined as finding, among all possibilities to match two probability measures, the one that minimizes the average transport cost for a given cost function $c$. While most of research assumes a given known cost function, recently the problem when the cost is unknown became of interest, motivated by various applications. This has been coined as ``inverse optimal transport'' in \cite{InverseSIam} and the authors aim to recover the cost function underlying migration flows from observed OT plans. Another possible application concerns the recovery of consumer preferences (Hotelling's location problem, see \cite[Section~5.1]{GalichonBook}): Suppose that we observe where people shop a certain item, i.e., a matching between people and shops, and we assume that this selection of a shop is driven by some unknown cost function reflecting the preferences of consumers. From which information can we potentially recover this cost function? We discuss and formalize this example in more detail in \autoref{subsubsec:economical_example}. As a final motivating example we mention the recovery of the internal cost to match protein assemblies in the compartment of the cell. With modern super-resolution microscopy it is possible to visualize the locations of such different protein complexes (e.g. receptors, transporters or other functional units) at nanoscale accuracy. For understanding protein function and communication it is important to mathematically describe and model how the cell organizes this spatial formation, often denoted as colocalization. It has been shown recently that OT plans based on Euclidean cost, estimated from the data, can quantify colocalization in a biological meaningful way \cite{tameling2021colocalization}. Instead of assuming a fixed (hypothetical) cost function, we speculate that it might be possible even to recover the corresponding cost function (encoding the internal organization of the cell) from such super-resolution images (i.e. the OT marginals) together with additional measurements of the total energy (i.e. the corresponding total OT costs).

Since the pioneering work of \cite{InverseSIam} the topic of inverse OT has undergone a flourishing development, see \autoref{subsec:state_of_art}. In all these papers it is assumed that the OT plan is observable (i.e.\ known). From this the corresponding cost is to be recovered, typically by a Bayesian approach specifying a prior on the set of cost matrices \cite{InverseSIam}. Obviously, this is a severely ill-posed problem in the sense that from the given OT plans and such prior the cost cannot be identified, in general. We take this problem of non-identifiability as a starting point to investigate the more general question under which conditions the cost function will be (uniquely) identifiable from certain information in the solution of the OT problem, see \autoref{subsec:contrib}.

OT, as a convex optimization problem, admits a dual formulation. Therefore, when solving it, we can obtain three objects: the optimal transport plan (OP) (the minimizer of the primal problem), the potentials (the maximizer of the dual problem), and the optimal total cost (OC) (the optimal value of the objective function). Each of these objects has its own interpretation and may be accessible or not, depending on the specific application. In the case under study, the practitioner may observe one or several of these objects. For example, in Hotelling's location problem the OT potentials are assumed to be known, whereas in the colocalization problem an ensemble of pairs of marginals is observed together with the optimal total costs. In particular, we will always assume that the marginals are observed, either explicitly or implicitly through the OPs.

\subsection{Contributions and Organization of the Work} \label{subsec:contrib}

The main contribution of this work is to establish sufficient conditions under which the cost in the inverse OT problem is identifiable, i.e., the cost (in a suitable space of cost functions) is uniquely determined by (a combination of) the abovementioned input objects. This will set some fundaments for further investigation how the cost function of a system can be recovered algorithmically as well as in a statistical context.

\autoref{subsec:state_of_art} provides a general overview of the inverse OT problem. \autoref{section:notation} introduces the notation used throughout the paper. In \autoref{section:inverseOT}, the standard and inverse OT problems are presented. \autoref{section:identi} contains the main results of this article. \autoref{subsec:convex_cost} provides sufficient identifiability conditions for convex costs and \autoref{subsec:concave_cost} does the same for concave costs. Finally, \autoref{sec:real_line} focuses on the univariate case where more refined results can be obtained.

The main results are presented in a decreasing order of assumed prior knowledge. Specifically, \autoref{thm:convex_cost_pi_dual_sol_known} (resp.\ \autoref{thm:concave_cost_dual_sol_known}) provides a uniqueness result for convex costs (resp.\ concave costs), assuming knowledge of the plans, optimal total costs, and potentials. In both cases, the cost function is determined within the union of the ranges of the gradients of the OT potentials. As knowing the OT potentials is not always possible, in \autoref{thm:same_dual_sol}, we derive sufficient conditions for continuous cost functions to identify the OT potentials from the OT cost values. In particular, \autoref{thm:same_dual_sol} assumes uniqueness of OT potentials (up to additive constants). By employing \cite[Corollary~2.7]{del2021central} for convex costs and \autoref{lemma:concave_cost_unique_potentials_supp} for concave costs, sufficient conditions for the mentioned uniqueness of potentials are obtained. These results enable us to formulate sufficient identifiability conditions for the cost function in terms of marginal probability measures -- refer to \autoref{thm:convex_cost_pi_unknown} for the convex cost case and \autoref{thm:concave_cost_dual_sol_unknown} for the concave cost case.

In the univariate case, our results can be sharpened further and we study in depth identifiability for convex cost functions and location-scale families. Notably, \autoref{thm:convex_cost_onedim_g_transform} links identifiability of the cost function with injectivity of an integral transform (namely the $g$-transform \eqref{eq:g_transform_def}) and the well known problem of statistical completeness in location-scale families \cite{Mattner1992}.

\subsection{State of the Art of Inverse Optimal Transport} \label{subsec:state_of_art}

As highlighted by \cite{andrade2023sparsistency}, the problem of estimating the cost function is related to ground metric learning, for which there are also OT based proposals \cite{zhang2023manifold,Huizing2021UnsupervisedGM}. However, the latter is generally simpler, as it does not involve a global matching between sets of points. That is, the inverse OT problem requires (as an input) pairs of probability distributions instead of pairs of data points. Until now, this problem has mainly been dealt with in machine learning (see \cite{Dupuy2016EstimatingMA,InverseJML,InversePMLR,InverseSIam,InverseCarlier}) and the focus has been on estimating the cost function from observable entropic transport plans, i.e., when the Kullback-Leibler divergence is added as a penalty to the objective function. The estimation proposed by \cite{Dupuy2016EstimatingMA} involves maximizing a likelihood function. However, such function is not strictly convex. Therefore, multiple maxima may exist, in general. \cite{InverseCarlier} formulated a modified problem that incorporates a Lasso penalty that enforces sparsity of the cost function, which is a desirable property in economic applications where agents only match to certain features. With this penalty the optimization problem becomes strictly convex and the argmax is a singleton (see \cite[Section~3.1]{andrade2023sparsistency}). However, this does not imply identifiability on the regularized optimal transport cost.

In summary, all the previously mentioned works provide (algorithmic) estimators of the cost function from a sample in different scenarios. However, there is no rigorous analysis under which situations the inverse OT problem is identifiable and the aim of this work is to fill this gap. Indeed, this is a minimal condition for the convergence of any algorithm to the true underlying cost function.

\subsection{Notation} \label{section:notation}

Let $\R^d$ be equipped with the Euclidean norm $\norm{\argdot}$ and the standard scalar product $\inner{\argdot}{\argdot}$. Denote with $\calP(\R^d)$ the set of Borel probability measures and for $p \geq 1$ with $\calP_p(\R^d)$ the set of probability measures with $p$-moment, i.e., all $\mu \in \calP(\R^d)$ such that $\int \norm{x}^p \de{\mu(x)} < \infty$. Furthermore, we equip $\calP_p(\R^d)$ with the $p$-Wasserstein distance $\W_p$ and the topology induced by it, see also \cite[Chapter~6]{villani2008optimal}. Denote with $\lambda_d$ the Lebesgue measure on $\R^d$. A probability measure $\mu \in \calP(\R^d)$ that is absolutely continuous w.r.t.\ the Lebesgue measure is denoted as $\mu \ll \lambda_d$ and $\supp \mu$ is the (closed) support of $\mu$. We say that $\mu$ has negligible boundary if $\lambda_d(\supp \mu \setminus \opint [\supp \mu]) = 0$, where $\opint A$ denotes the interior of a set $A \subseteq \R^d$. We call a Borel measurable set $A \subseteq \R^d$ rectifiable of dimension $d-1$ if it can be covered using countable many $(d-1)$-dimensional Lipschitz submanifolds of $\R^d$. For two probability measures $\mu, \nu \in \calP(\R^d)$, write $[\mu - \nu]_+ - [\nu - \mu]_+$ for the Jordan decomposition of $\mu - \nu$ into the difference of two non-negative mutually singular measures.

Denote with $\calC(\R^d)$ the set of all continuous functions on $\R^d$. For a convex function $h : \R^d \to \R$, let $h^*$ be the convex conjugate of $h$ \cite{rockafellar2009variational}, i.e.,
\begin{equation*}
	h^* : \R^d \to \R\,, \quad y \mapsto \sup_{x \in \R^d} \{ \inner{x}{y} - h(x) \}\,.
\end{equation*}
If $h$ is strictly convex and differentiable, note that $h^*$ is differentiable with $\nabla h^* = [\nabla h]^{-1}$. Further, for a concave function $l : [0, \infty) \to [0, \infty)$, define the concave conjugate of $h \defeq l \circ \norm{\argdot}$ by
\begin{equation*}
	h^* : \R^d \to \R\,, \quad y \mapsto -(-l)^*(-\norm{y})\,,
\end{equation*}
where $-l$ is convex and extended to $(-\infty, 0)$ by setting it to $\infty$.

\section{From Forward to Inverse Optimal Transport}\label{section:inverseOT}

In this section, we first introduce the problem of OT in its usual formulation. Then, we will explain the inverse problem.

\subsection{Optimal Transport}

Let $c : \R^d \times \R^d \to [0, \infty)$ be a continuous cost function. The OT problem between two probability measures $\mu \in \calP(\R^d)$, $\nu \in \calP(\R^d)$ with respect to (w.r.t.) the cost function $c$ is defined as
\begin{equation} \label{eq:OT}
	\OT_c(\mu,\nu) = \min_{\pi \in \Pi(\mu,\nu)} \int c \de{\pi}\,,
\end{equation}
where $\Pi(\mu,\nu)$ represents the set of probability measures on $ \calP(\R^d \times \R^d)$ with marginals $\mu$ and $\nu$. We call a minimizer $\pi$ an OT plan. If existent, a measurable map $T : \R^d \to \R^d$ with $T_{\#} \mu = \nu$ is called the OT map if $(\id, T)_{\#} \mu$ is an OT plan.

Furthermore, \eqref{eq:OT} admits the following dual formulation
\begin{equation} \label{eq:dual}
	\OT_c(\mu,\nu)=\max_{(f,g) \in \Phi_c(\mu, \nu)} \int f \de\mu +\int g \de\nu,
\end{equation}
where
\begin{equation*}
	\Phi_c(\mu, \nu) \defeq \{(f,g)\in L^1(\mu)\times L^1(\nu): \ f(x)+g(y)\leq c(x,y)\ \text{for all } x,y\in \R^d \}\,.
\end{equation*}
Additionally, we denote with $\calS_{c}(\mu, \nu)$ the set of optimal potentials for $\OT_c(\mu, \nu)$, i.e.,
\begin{equation*}
	\calS_c(\mu, \nu) \defeq \{ (f, g) \in \Phi_c(\mu, \nu) : \int f\de{\mu} + \int g\de{\nu} = \OT_c(\mu, \nu) \}\,.
\end{equation*}

\subsection{Inverse Optimal Transport}

Let $\calH$ be some class of cost functions. Given an underlying but unknown cost function $c \in \calH$, suppose that we observe a set of marginals $\{ (\mu^{(v)}, \nu^{(v)}) \}_{v \in \calV} \subseteq \calP(\R^d) \times \calP(\R^d)$ indexed by some non-empty set $\calV$. Moreover, for each $v \in \calV$ we are given (a subset of) the OT information between $\mu^{(v)}$ and $\nu^{(v)}$ w.r.t.\ $c$: \begin{enumerate}[(i)] \item\label{enum:ot_cost} The corresponding OT costs $\alpha^{(v)} \in \R$, \item\label{enum:ot_plan} OT plan $\pi^{(v)} \in \Pi(\mu^{(v)}, \nu^{(v)})$ and \item\label{enum:ot_pot} optimal potentials $(f^{(v)}, g^{(v)}) \in \Phi_c(\mu^{(v)}, \nu^{(v)})$ such that \end{enumerate}
\begin{equation*}
	\alpha^{(v)} = \OT_c(\mu^{(v)}, \nu^{(v)}) = \int c \de{\pi^{(v)}} = \int f^{(v)} \de{\mu^{(v)}} + \int g^{(v)} \de{\nu^{(v)}}\,.
\end{equation*}
The goal of inverse OT is to infer the underlying cost function $c$ from (a subset of) the above OT information \ref{enum:ot_cost}-\ref{enum:ot_pot}.

\subsubsection{Economical example \texorpdfstring{\small (Hotelling's location model \cite[Section 5.1]{GalichonBook})}{(Hotelling's location model)}} \label{subsubsec:economical_example}
Let $\calX$ represent a city and $\calY$ a set of shops that sell some item $I$. We are given probability distributions $\mu$ and $\nu$ on $\calX$ and $\calY$, that is the distribution of citizens and items, respectively. For each shop $y \in \calY$ we observe the price $z(y)$ of item $I$ and for each citizen $x \in \calX$ the effort $e(x)$ they make to buy the item (e.g. the time they spent to get to the shop). Furthermore, we have access to a map $T$ that gives for each citizen $x \in \calX$ the shop $y = T(x)$ they buy the item $I$ at. The price $z$ and effort $e$ could be obtained through a survey, whereas $T$ might be available by the use of tracking data.

We assume that there exists some underlying cost function $c$ such that: Each citizen $x \in \calX$ chooses the shop such that its effort is minimal in the sense that $e(x) = \inf_{y \in \calY} \{ c(x,y) - z(y) \}$, and $T$ is the OT map w.r.t.\ $c$ between $\mu$ and $\nu$. Then, it holds that $(e, z)$ are the optimal potentials,
\begin{equation*}
    \int_{\calX} e(x) \de{\mu(x)} + \int_{\calY} z(y) \de{\nu(y)} = \int_{\calX} c(x, T[x]) \de{\mu(x)}\,.
\end{equation*}
Knowing $c$ means understanding the preferences of consumers when choosing a shop. With this information, an analyst can advise a store on the necessary changes to increase its profitability.

In the above setting, we effectively have access to the full OT information \ref{enum:ot_cost}-\ref{enum:ot_pot}. Note that this might not always be possible. For example, tracking data might not be available because of privacy concerns and then only the optimal potentials $(e, z)$ can be observed through a survey. Vice versa, a survey might be too expensive or difficult to run and only the map $T$ can be obtained.

\section{Identifiability}\label{section:identi}

In the following, we investigate under which conditions and combinations of \ref{enum:ot_cost}-\ref{enum:ot_pot} $c$ is identifiable in $\calH$, i.e., when the (sub-)set of OT information uniquely determines $c$ in $\calH$. To this end, we consider classes of cost functions $\calH$ such that the OT w.r.t.\ any cost function from said class has certain structural properties that we can exploit. We will always assume that $c$ is induced by a function $h : \R^d \to [0, \infty)$, i.e., $c$ can be written as $c(x, y) = h(x - y)$. Then, identifiability of $c$ reduces to that of the inducing $h$. First, we observe that the knowledge of the OT cost is a necessary condition for achieving identifiability already in the univariate case and for convex function classes.

\subsection{Inverse Optimal Transport Knowing only the Optimal Transportation Plans is not Identifiable}

It has been noted in \cite{andrade2023sparsistency,InverseCarlier} that inverse OT is an ill-posed problem (in the sense of non-identifiability) when we aim to obtain the cost function based solely on the knowledge of the OT plans.

\begin{example}[See Section 2.2.\ in \cite{villani2003topics}]
	Let $\calH$ be the class of functions $c:\R\times \R\to [0,\infty) $ such that $c(x,y)=h(x-y)$ where $h:\R\to [0,\infty)$ is continuous and strictly convex. Then, for any pair $ (\mu, \nu) \in \calP(\R) \times \calP(\R)$ the monotone coupling $(F_\mu^{-1}, F_\nu^{-1})_{\#}(\restr{\lambda_1}{[0,1]})$ is optimal for every $c \in \calH$, where $F^{-1}_\mu$, $F^{-1}_\nu$ are the quantile functions of $\mu$, $\nu$ and $\restr{\lambda_1}{[0,1]}$ denotes the restriction of $\lambda_1$ to the unit interval. As a consequence, no cost function $c \in \calH$ is identifiable from the OT plans only.
\end{example}

\begin{example}
    It holds for any pair $(\mu, \nu) \in \calP(\R^d) \times \calP(\R^d)$ that if $\pi^* \in \Pi(\mu, \nu)$ is an OT plan w.r.t.\ $c$ it follows that $\pi^*$ is also optimal w.r.t.\ $c + k$ for any constant $k \in \R$. Hence, when observing only the OT plans, a cost function $c$ cannot be identifiable in $\calH$ if $c + k \in \calH$ for some $k \in \R$.
\end{example}

These elementary examples indicate that only knowledge of the OT plan (and also its marginals) is not sufficient to identify the cost function. However, the situation changes fundamentally if the optimal potentials are known: In the primal problem, the cost function $c$ influences the objective function, whereas in the dual problem, it affects the feasibility set. This leads to the possibility of obtaining the OT cost $\OT_c(\mu, \nu)$ through knowledge of the marginal probabilities $(\mu, \nu)$ and optimal potentials $(f,g)$ via the identity
\begin{equation}
	\label{fromPotentialsToCost}
	\OT_c(\mu, \nu) = \int f \de{\mu} + \int g\de{\nu}\,.
\end{equation}

\subsection{Convex Cost Functions} \label{subsec:convex_cost}

First, we consider cost functions that are induced by convex functions. Let $\calH$ be the set of strictly convex and differentiable functions $h : \R^d \to [0, \infty)$ that satisfy conditions \textup{(H2-3)} from Gangbo \& McCann \cite{GaMc}. That is,
\begin{itemize}
	\item[\textup{(H2)}] for height $r > 0$ and angle $\theta \in (0,\pi)$ it holds for all points $x \in \R^d$ that are far enough from the origin that there exists a cone
	\begin{equation*}
		\{ y \in \R^d : \norm{y-x} \norm{z} \cos(\theta / 2) \leq \inner{z}{y-x} \leq r \norm{z} \} \qquad \text{where } z \in \R^d \setminus \{0\}\,,
	\end{equation*}
	with vertex at $x$ on which $h$ attains its maximum at $x$;
	\item[\textup{(H3)}] $h(x) / \norm{x} \to \infty$ as $\norm{x} \to \infty$.
\end{itemize}
Furthermore, denote for $p \geq 1$ with $\calH_p$ the subset of $h \in \calH$ that are dominated by $\norm{\argdot}^p$, i.e., $h(x) \leq A \norm{x}^p + B$ for some constants $A,B > 0$ and all $x \in \R$. For cost functions induced by $\calH$, Gangbo and McCann showed that the OT between certain marginals is always induced by a unique OT map which has an explicit formula in terms of the inducing function and an optimal potential. Supposing that two cost functions have the same OT map, we can exploit this to derive conditions for identifiability. First, we treat the case where we know all input objects \ref{enum:ot_cost}-\ref{enum:ot_pot}, i.e., OT costs as well as OT plans and optimal potentials.

\begin{theorem}[Identifiability for strictly convex costs] \label{thm:convex_cost_pi_dual_sol_known}
Let $c_i: \R^d \times \R^d \to [0, \infty)$ be induced by $h_i \in \calH$, i.e., $c_i(x,y)=h_i(x-y)$, for $i \in \{1, 2\}$. \\
If there exists a non-empty set $$\{(\mu^{(v)},\nu^{(v)})\}_{v\in \calV}\subset \calP(\R^d) \times \calP(\R^d)$$
such that for every $v\in \calV$ it holds that $\mu^{(v)} \ll \lambda_d$ and there exist $\pi^{(v)}$ together with $(f^{(v)}, g^{(v)})$, optimal plan and optimal potentials for $\mu^{(v)}$ and $\nu^{(v)}$ w.r.t.\ $c_1$ and $c_2$ simultaneously, then $(\nabla h_1^*)(y)= (\nabla h_2^*)(y)$ for every
\begin{equation*}
    y \in \calX \defeq \bigcup_{v \in \calV} \{ \nabla f^{(v)}(x): x \in \opint(\supp\mu^{(v)}) \}\,.
\end{equation*}
As a consequence, for any open and connected set $\calX' \subseteq \R^d$ such that $\nabla h_1(\calX') \subseteq \calX$, it holds that $h_1 = h_2 + k$ on $\calX'$ for some constant $k \in \R$.

Moreover, if there exists $v \in \calV$ such that $\OT_{c_1}(\mu^{(v)}, \nu^{(v)}) = \OT_{c_2}(\mu^{(v)}, \nu^{(v)})$ together with $$\{ x - y : x \in \supp \mu^{(v)},\, y \in \supp \nu^{(v)}\} \subseteq \calX'\,,$$ we even have $k = 0$.
\end{theorem}
\begin{proof}
	For $v \in \calV$, \cite[Theorem~1.2]{GaMc} asserts that we have the following OT maps for $\mu^{(v)}$ and $\nu^{(v)}$ w.r.t.\ the cost $c_i$, $i \in \{1, 2\}$,
	\begin{equation*}
		T^{(v)}_i(x) = x - (\nabla h_i^*)[\nabla f^{(v)}(x)]\, .
	\end{equation*}
	In particular, they are unique $\mu^{(v)}$-a.s.\ and induce the unique OT plan. Hence, by assumption we obtain $\mu^{(v)}$-a.s.\ that $T^{(v)}_1=T^{(v)}_2$. Therefore, for $\mu^{(v)}$-a.e.\ $x$ it holds
	\begin{equation*}
		(\nabla h_1^*)[\nabla f^{(v)}(x)]= (\nabla h_2^*)[\nabla f^{(v)}(x)],
	\end{equation*}
	so that $ (\nabla h_1^*)(y)= (\nabla h_2^*)(y)$, for all
	\begin{equation*}
		y \in \{ \nabla f^{(v)}(x): x \in \opint(\supp \mu^{(v)})\}\,.
	\end{equation*}
	Now, suppose that $\nabla h_1(\calX')\subseteq \calX$ for some open and connected $\calX' \subseteq \R^d$, then for any $x\in \calX'$ it holds $\nabla h_1(x)\in \calX$ and
	\begin{equation*}
		x=(\nabla h_1^*)[(\nabla h_1)(x)]= (\nabla h_2^*)[(\nabla h_1)(x)]\,.
	\end{equation*}
	Notably, this implies that $\nabla h_2^* \circ \nabla h_1 = \id$ on $\calX'$ which can only happen if $(\nabla h_2^*)^{-1}=\nabla h_1$. As it holds that $(\nabla h_2^*)^{-1}=\nabla h_2$, we see that $\nabla h_2 = \nabla h_1$ on $\calX'$ and conclude $h_1 = h_2 + k$ for some $k \in \R$. If we further assume that $\{ x - y : x \in \supp \mu^{(v)},\, y \in \supp \nu^{(v)}\} \subseteq \calX'$ for some $v \in \calV$, we have that $c_1 = c_2 + k$ on $\supp \mu^{(v)} \times \supp \nu^{(v)}$. Hence, it follows that
	\begin{equation*}
		\OT_{c_2}(\mu^{(v)}, \nu^{(v)}) = \OT_{c_1}(\mu^{(v)}, \nu^{(v)}) = \OT_{c_2+k}(\mu^{(v)}, \nu^{(v)}) = \OT_{c_2}(\mu^{(v)}, \nu^{(v)}) + k\,,
	\end{equation*}
	which implies $k = 0$.
\end{proof}

\begin{remark}
    Note that in \autoref{thm:convex_cost_pi_dual_sol_known} only the first optimal potential $f^{(v)}$ needs to be known explicitly and not the second one $g^{(v)}$. This is also true for \autoref{thm:concave_cost_dual_sol_known}. As $f^{(v)} = \inf_{y \in \R^d} \{ c(\argdot,y) -g^{(v)}(y) \}$ is the $c$-transform of $g^{(v)}$, we use $g^{(v)}$ only implicitly. Note that as $c$ is unknown, we cannot use this relationship explicitly to obtain $f^{(v)}$ from $g^{(v)}$ (and vice versa).
\end{remark}

The approach used in \autoref{thm:convex_cost_pi_dual_sol_known} always requires knowledge of one optimal potential as the OT map depends on it. To circumvent this, we give a criterion where the optimal potentials for certain marginals do not depend on the specific underlying cost function. Namely, we assume that the marginals are dense in an open set. However, due to the unboundedness of the cost functions induced by $\calH$, open sets of $\calP(\R^d)$ equipped with the usual narrow topology (weak convergence) are not well suited. To this end, we take the marginals to belong to the $p$-Wasserstein space $(\calP_p(\R^d), \W_p)$ for some $p \geq 1$. This way, we make sure that the OT cost between such marginals is always finite w.r.t. cost functions induced by $\calH_p$.

\begin{lemma} \label{thm:same_dual_sol}
Let $c_1,\, c_2 : \R^d \times \R^d \to [0, \infty)$ be continuous cost functions with constants $A,B > 0$ such that
\begin{equation*}
\max[ c_1(x,y), c_2(x,y) ] \leq A \norm{x-y}^p + B \qquad \text{for all } x,\, y \in \R^d\,.
\end{equation*}
Assume that there exist a dense subset $\{\mu^{(v)}\}_{v\in \calV}$ of an non-empty open set $\mathcal{U} \subseteq \calP_p(\R^d)$ and $\nu \in \calP_p(\R^d)$ such that
\begin{equation*}
	\OT_{c_1}(\mu^{(v)},\nu) = \OT_{c_2}(\mu^{(v)},\nu) \qquad \text{for all $ v \in \calV $.}
\end{equation*}
Then it holds for any $v \in \calV$ such that the optimal potentials $(f_i^{(v)}, g_i^{(v)})$ w.r.t.\ $c_i$, $i \in \{1,2\}$, are unique (up to additive constants), that they are equal and thus optimal for $c_1$ and $c_2$ simultaneously.
\end{lemma}
\begin{proof}
Let $w \in \calV$ be such that the optimal potentials are unique. Denote with $\calC_{0}(\R^d)$ the set of continuous and compactly supported functions that integrate (w.r.t.\ $\lambda_d$) to $0$. For $\phi \in \calC_{0}(\R^d)$, let $\xi_\phi$ be the measure on $\R^d$ with Lebesgue density $\phi$. Since $\calU$ is open and $\xi_\phi$ compactly supported, there exists a $t_0 > 0$ small enough such that $\mu^{(w)}_t \defeq \mu^{(w)} + t \xi_\phi \in \mathcal{U}$ for all $t \in [0, t_0] $. As $\{\mu^{(v)}\}_{v\in \calV}$ is dense, for each $t \in [0, t_0]$ there exists a net $\{ \mu_{n} \}_{n\in\N} \subset \{\mu^{(v)}\}_{v\in\calV}$ such that $\W_p(\mu_n, \mu^{(w)}_t) \to 0$ as $n \to \infty$. By the continuity of the OT functional \cite[Theorem~5.20]{villani2008optimal}, it follows that the function
$$ \Psi : [0, t_0] \to \R\,, \qquad t \mapsto \OT_{c_1}(\mu^{(w)}_t,\nu) - \OT_{c_2}(\mu^{(w)}_t,\nu)\,, $$
is constant equal to $0$. \cite[Proposition~4.8]{Santambrogio2017} implies that the right derivative (first variation) of $\Psi$ at $0$ is
$\int (f_{1}^{(w)}-f_{2}^{(w)}) \phi \de \lambda_d \, . $ Hence,
$$ 0 = \int (f_{1}^{(w)}-f_{2}^{(w)}) \phi \de{\lambda_d} \qquad \text{for all } \phi \in \calC_0(\R^d)\,.$$
As $\calC_0(\R^d)$ is dense w.r.t.\ the weak topology the proof is complete.
\end{proof}

We know, from \eqref{fromPotentialsToCost}, that knowledge of the optimal potentials determines the value of the OT cost. \autoref{thm:same_dual_sol} shows the reverse implication. That is, from sufficient knowledge of transportation costs, the transportation potentials are automatically determined. The main assumption of \autoref{thm:same_dual_sol} is that the OT potentials are unique (up to constant shift). Using the sufficient conditions provided in \cite[Corollary~2.7]{del2021central} for convex cost functions, we obtain the following result. Note that more general uniqueness conditions for optimal potentials are given in \cite{staudt2022uniqueness}.

\begin{corollary}\label{thm:convex_cost_same_dual_sol}
Let $c_i : \R^d \times \R^d \to [0, \infty)$ be induced by $h_i \in \calH_p$, i.e., $c_i(x,y)=h_i(x-y)$, for $i \in \{1, 2\}$.
Assume that there exists a dense subset $\{(\mu^{(v)}, \nu^{(v)}) \}_{v\in \calV}$ of an non-empty open set of $\calP_p(\R^d) \times \calP_p(\R^d)$ such that
\begin{equation*}
	\OT_{c_1}(\mu^{(v)},\nu^{(v)}) = \OT_{c_2}(\mu^{(v)},\nu^{(v)}) \qquad \text{for all } \ v \in \calV\,.
\end{equation*}
Recalling the definition of negligible boundary from \autoref{section:notation}, define the set
\begin{equation} \label{eq:def_calv_lambda_d}
	\calV(\lambda_d) \defeq \{ v \in \calV \mid \mu^{(v)} \ll \lambda_d \text{ has negligible boundary and $\supp \mu^{(v)}$ is connected}\}\,.
\end{equation}
Then, it follows for all $v \in \calV(\lambda_d)$ that $\calS_{c_1}(\mu^{(v)}, \nu^{(v)}) = \calS_{c_2}(\mu^{(v)}, \nu^{(v)})$ and the dual solution $(f^{(v)}, g^{(v)}) \in \calS_{c_1}(\mu^{(v)}, \nu^{(v)})$ is unique (up to additive constants).
\end{corollary}
\begin{proof}
It follows by \cite[Corollary~2.7]{del2021central} that for $v \in \calV(\lambda_d)$ the optimal potentials $(f_i^{(v)}, g_i^{(v)}) \in \calS_{c_i}(\mu^{(v)}, \nu^{(v)})$, $i \in \{1, 2\}$, are unique (up to additive constants). An application of \autoref{thm:same_dual_sol} yields the assertion.
\end{proof}

Under the assumptions of \autoref{thm:convex_cost_same_dual_sol}, we can get rid of the condition in \autoref{thm:convex_cost_pi_dual_sol_known} involving the knowledge of the optimal potentials. As a consequence, we obtain an identifiability result where only the OT plans and the OT costs for an open set of marginals need to be known.

\begin{corollary}
\label{thm:convex_cost_pi_known}
	Let the setting of \autoref{thm:convex_cost_same_dual_sol} hold. If for every $v\in \calV(\lambda_d)$ the (unique) OT plan is the same for both costs $c_1$ and $c_2$, then $(\nabla h_1^*)(y)= (\nabla h_2^*)(y)$ for every
	\begin{equation*}
		y \in \calX \defeq \bigcup_{v \in \calV(\lambda_d)} \{ \nabla f^{(v)}(x): x \in \opint(\supp\mu^{(v)}) \}\,.
	\end{equation*}
	In particular, for every open and connected set $\calX' \subseteq \R^d$ with $\nabla h_1(\calX') \subseteq \calX$ there exists $k \in \R$ such that $h_1 = h_2 + k$ on $\calX'$. Moreover, if there exists $v \in \calV$ such that
	$$\{ x - y : x \in \supp \mu^{(v)},\, y \in \supp \nu^{(v)}\} \subseteq \calX'\,,$$
	we even have $k = 0$.
\end{corollary}
\begin{proof}
	Follows from \autoref{thm:convex_cost_pi_dual_sol_known} applied to the subset $\{(\mu^{(v)}, \nu^{(v)})\}_{v \in \calV(\lambda_d)}$.
\end{proof}
\autoref{thm:convex_cost_pi_known} provides sufficient conditions for the identifiability of the cost when the OPs are known. This is summarized in the following result.

\begin{corollary}
	Let $c_i: \R^d \times \R^d \to [0, \infty)$ be induced by $h_i \in \calH_p$, i.e., $c_i(x,y)=h_i(x-y)$, for $i \in \{1, 2\}$. \\
    Assume that there exist transport plans $\{\pi^{(v)}\}_{v\in \calV}$ such that the set of marginals $\{(\mu^{(v)},\nu^{(v)})\}_{v\in \calV}$ is a dense subset of a non-empty open set of $\calP_p(\R^d) \times \calP_p(\R^d)$ and for every $v \in \calV$,
    \begin{equation*}
			\OT_{c_1}(\mu^{(v)},\nu^{(v)}) = \int c_1 \de{\pi}^{(v)}= \int c_2 \de{\pi}^{(v)}= \OT_{c_2}(\mu^{(v)},\nu^{(v)})\,.
	\end{equation*}
    With $\calV(\lambda_d)$ as in \eqref{eq:def_calv_lambda_d}, it follows that $(\nabla h_1^*)(y)= (\nabla h_2^*)(y)$ for all
	\begin{equation*}
		y \in \calX \defeq \bigcup_{v \in \calV(\lambda_d)} \{ \nabla f^{(v)}(x): x \in \opint(\supp\mu^{(v)}) \}\,.
	\end{equation*}
    In particular, let $\calX' \subseteq \R^d$ be an open and connected set with $\nabla h_1(\calX') \subseteq \calX$, then it holds that $h_1 = h_2 + k$ on $\calX'$ for some constant $k \in \R$. Moreover, if there exists $v \in \calV$ such that $$\{ x - y : x \in \supp \mu^{(v)},\, y \in \supp \nu^{(v)}\} \subseteq \calX'\,,$$ we even have $k = 0$.
\end{corollary}

Finally, we give a condition that reduces knowledge about the OT plan to the unique and shared optimal potentials of \autoref{thm:convex_cost_same_dual_sol}. In particular, this yields an identifiability result that only requires the marginals (and the specific choice of the class of cost functions).

\begin{theorem}\label{thm:convex_cost_pi_unknown}
	Let the setting of \autoref{thm:convex_cost_same_dual_sol} hold. If for every $v\in \calV(\lambda_d)$ the OT plan between $\mu^{(v)}$ and $\nu^{(v)}$ for the cost
	\begin{equation*}
		c_{\calV}(x,y)=\sup_{v\in \calV(\lambda_d)} \{ f^{(v)}(x)+ g^{(v)}(y) \}
	\end{equation*} is unique, then, for every $v\in \calV(\lambda_d)$, the unique OT map is the same for $c_1$ and $c_2$. Moreover, it holds that $(\nabla h_1^*)(y)= (\nabla h_2^*)(y)$ for all
	\begin{equation*}
		y \in \calX \defeq \bigcup_{v \in \calV(\lambda_d)} \{ \nabla f^{(v)}(x): x \in \opint(\supp\mu^{(v)}) \}\,.
	\end{equation*}
	In particular, for any open and connected set $\calX' \subseteq \R^d$ such that $\nabla h_1(\calX') \subseteq \calX$, there exists a constant $k \in \R$ such that $h_1 = h_2 + k$ on $\calX'$. Moreover, if there exists $v \in \calV$ such that $$\{ x - y : x \in \supp \mu^{(v)},\, y \in \supp \nu^{(v)}\} \subseteq \calX'\,,$$ we even have $k = 0$.
\end{theorem}
\begin{proof}
	We show that in this setting the OT plan (or map) between $\mu^{(v)}$ and $\nu^{(v)}$ for $v \in \calV(\lambda_d)$ is the same w.r.t.\ both costs $c_1$ and $c_2$. The additional assertions then follow from \autoref{thm:convex_cost_pi_dual_sol_known}.

	For $w \in \calV(\lambda_d)$, \cite[Theorem~1.2]{GaMc} assert that we have the following OT maps for $\mu^{(w)}$ and $\nu^{(w)}$ w.r.t.\ the cost $c_i$, $i \in \{1, 2\}$,
	\begin{equation*}
		T^{(w)}_i(x) = x - (\nabla h_i^*)[\nabla f^{(w)}(x)]\,.
	\end{equation*}
	Since by optimality \cite[Theorem 5.10]{villani2008optimal} it holds for $\mu^{(w)}$-a.e.\ $x$ that
	\begin{equation}\label{eq:dual_map_optimality}
		f^{(w)}(x)+ g^{(w)}(T^{(w)}_1(x))=h_1(x-T^{(w)}_1(x))\,,
	\end{equation}
	and $(f^{(w)}, g^{(w)}) \in \Phi_{c_2}(\mu^{(w)}, \nu^{(w)})$, it follows for $\mu^{(w)}$-a.e.\ $x$ that
	\begin{equation*}
	h_1(x-T^{(w)}_1(x))\le h_2(x - T^{(w)}_1(x))\,.
	\end{equation*}
	Recall the cost function
	\begin{equation*}
		c_{\calV}(x,y)=\sup_{v\in \calV(\lambda_d)} \{ f^{(v)}(x)+ g^{(v)}(y) \}\,,
	\end{equation*}
	and note that
	\begin{equation*}
		\Phi_{c_{\calV}}(\mu^{(w)}, \nu^{(w)}) \subset \Phi_{c_1}(\mu^{(w)}, \nu^{(w)}) \cap \Phi_{c_2}(\mu^{(w)}, \nu^{(w)})\,.
	\end{equation*}
	By definition, it holds $(f^{(w)},g^{(w)})\in \Phi_{c_{\calV}}(\mu^{(w)}, \nu^{(w)})$. Therefore,
	\begin{equation*}
		f^{(w)}(x)+ g^{(w)}(T^{(v)}_1(x)))\leq c_{\calV}(x,T^{(w)}_1(x))\leq h_1(x-T^{(w)}_1(x))\,,
	\end{equation*}
	for all $x \in \calX$, so that, via \eqref{eq:dual_map_optimality}, $c_{\calV}(x,T^{(w)}_1(x))= h_1(x-T^{(w)}_1(x))$, for $\mu^{(w)}$-a.e.\ $x$. Integrating we obtain
	\begin{equation*}
		\int c_{\calV}(x,T^{(w)}_1(x)) \de\mu^{(w)}(x) = \int h_1(x-T^{(w)}_1(x)) \de\mu^{(w)}(x) =\OT_{c_1}(\mu^{(w)},\nu^{(w)}),
	\end{equation*}
	but, since $\Phi_{c_{\calV}}(\mu^{(w)}, \nu^{(w)}) \subset \Phi_{c_1}(\mu^{(w)}, \nu^{(w)})$,
	\begin{align*}
		\OT_{c_{\calV}}(\mu^{(w)},\nu^{(w)}) &= \sup_{(f,g)\in \Phi_{c_{\calV}}(\mu^{(w)},\nu^{(w)})} \int f \de\mu^{(w)} +\int g \de\nu^{(w)} \\
		&\leq \sup_{(f,g)\in \Phi_{c_1}(\mu^{(w)},\nu^{(w)})} \int f \de\mu^{(w)} +\int g \de\nu^{(w)} \\
		&= \int f^{(w)} \de\mu^{(w)} +\int g^{(w)} \de\nu^{(w)}=\OT_{c_1}(\mu^{(w)},\nu^{(w)})\,.
	\end{align*}
	Therefore, $\OT_{c_{\calV}}(\mu^{(w)},\nu^{(w)})=\OT_{c_1}(\mu^{(w)},\nu^{(w)})$ and $(\id \times T^{(w)}_1)_{\#} \mu^{(w)} $ is an OT plan for $\OT_{c_{\calV}}(\mu^{(w)},\nu^{(w)})$. The same holds for $(\id \times T^{(w)}_2)_{\#} \mu^{(w)} $. Under the assumption that $c_{\calV}$ admits a unique OT plan, it holds $\mu^{(w)}$-a.e. that $T^{(w)}_1=T^{(w)}_2$.
\end{proof}

Next, we illustrate our identifiability results with two examples.

\begin{example}[Normal distributions]
	Let $c(x, y) = h(x-y) = \norm{x-y}^2$ be the squared Euclidean distance and note that $h \in \calH$. For vectors $a, b \in \R^d$ and symmetric, positive definite matrices $A,\, B \in \R^{d \times d}$, consider two multivariate normal distributions $\mu \defeq \calN(a, A)$ and $\nu \defeq \calN(b, B)$. In this setting, the OT problem is well-studied (see \cite{Dowson82,Olikin82} or \cite{Takatsu2011} and references therein) and upon defining the matrix $D \defeq A^{-1/2}(A^{1/2} B A^{1/2})^{1/2} A^{-1/2}$, the OT map can be written as
	\begin{equation*}
		T(x) = x - ([I - D] x + D a - b) = x - \frac{1}{2} \nabla f(x)\,,
	\end{equation*}
	where $f$ is the first optimal potential. Hence, the image of $\nabla f$ is an affine subspace of $\R^d$. Taking $A = \sigma_1^2 I$ and $B = \sigma_2^2 I$ for $\sigma_1,\,\sigma_2 > 0$ where $I \in \R^{d \times d}$ is the identity matrix, it follows that $D = \sigma_1^{-1} \sigma_2 I$ and thus
	\begin{equation*}
		\{ \nabla f(x) : x \in \opint \supp \mu\} = \begin{cases}
			\R^d & \sigma_1 \neq \sigma_2\,, \\
			2(a - b) & \sigma_1 = \sigma_2\,.
		\end{cases}
	\end{equation*}
	Consequently, \autoref{thm:convex_cost_pi_dual_sol_known} yields that the optimal potentials together with the OT plan for one pair of Gaussians with different variances identify the squared Euclidean distance as the cost in the class of cost functions induced by $\calH$. In the case of equal variances, we obtain identifiability if for every $x \in \R^d$ we observe the OT between normal distributions such that $x=a-b$. Note that in this case the OT map is just the translation $T(x) = x - a + b$.

	Furthermore, if we observe a dense subset of an open set of marginals that contain the above normal distributions, \autoref{thm:convex_cost_pi_known} yields identifiability in the case that we only know the OT plans and not the optimal potentials. In particular, for $\sigma_1 \neq \sigma_2$ we only need to know the OT plan between $\mu$ and $\nu$ (and not between the other marginals).
\end{example}

\begin{example}[Real line]
	Consider the case $d=1$ and a cost function $c$ that is induced by $h \in \calH$. For probability measures $\mu,\,\nu \in \calP(\R)$ such that $\mu \ll \lambda_1$, we have by \cite[Theorem~2.9]{Santambrogio2015} that the monotone map is optimal,
	\begin{equation*}
		T(x) = F^{-1}_\nu[F_\mu(x)] = x - (h')^{-1}[f'(x)] \,,
	\end{equation*}
	where $F_\nu^{-1}$ is the quantile function of $\nu$, $F_\mu$ the cumulative distribution function (c.d.f.) of $\mu$ and $f$ the first optimal potential. Hence, it holds that
	\begin{equation*}
		f'(x) = h'(x - F^{-1}_\nu[F_\mu(x)])\,.
	\end{equation*}
	Thus, \autoref{thm:convex_cost_pi_dual_sol_known} yields an identifiability criterion and, roughly speaking, the structure of $F_\nu^{-1} \circ F_\mu$ determines the contribution of the pair $(\mu, \nu)$ to the identifiability of the cost $c$. Notably, this is minimal for $F_\nu^{-1} \circ F_\mu = \id$, that is $\mu = \nu$.

	Assume now that the marginals are coming from a location-scale family which is generated by a strictly increasing and differentiable c.d.f.\ $G$. Then, it holds for $F_\mu = G_{a,b}$ and $F_\nu = G$, say, that
	\begin{equation*}
		f'(x) = h'([1 - b]x - a)\,.
	\end{equation*}
	In particular, \autoref{thm:convex_cost_pi_dual_sol_known} implies that $h$ is identifiable in $\calH$ if we only observe one pair of the location-scale family with $b \neq 1$. Recall that this requires knowledge of the OT cost and a joint OT plan. In \autoref{sec:real_line}, we will discuss this setting in the case where we only observe the OT costs and marginals.
\end{example}

\subsection{Concave Cost Functions of the Distance} \label{subsec:concave_cost}

Notably, Gangbo and McCann \cite{GaMc} also provide a formula of the OT map for concave costs of the distance. Denote with $\calL$ the set of strictly concave and differentiable functions $l : [0, \infty) \to [0, \infty)$ with $l(0) = 0$, and the induced cost functions with $\calL_{\norm{\argdot}} \defeq \{ l \circ \norm{\argdot} \mid l \in \calL\}$. As for $\calH$ and $\calH_p$, denote with $\calL_p$ the subset of functions in $\calL$ that are dominated by $\norm{\argdot}^p$. Since the common mass $\min(\mu, \nu)$ is not transported for probability measures $\mu$, $\nu$ in this setting, this map is to be understood as the OT map between the rest of the mass, that is $[\mu - \nu]_+$ and $[\nu - \mu]_+$. As for the convex case, we use this map to derive identifiability results. Again, we start with the case where we observe the full OT information.

\begin{theorem}[Identifiability for strictly concave costs of the distance] \label{thm:concave_cost_dual_sol_known}
	Let $c_i : \R^d \times \R^d \to [0,\infty)$ be induced by $h_i \defeq l_i \circ \norm{\argdot} \in \calL_{\norm{\argdot}}$, i.e., $c_i(x,y)=h_i(x-y)$, for $i \in \{1, 2\}$. If there exists a non-empty set
    \begin{equation*}
        \{(\mu^{(v)},\nu^{(v)})\}_{v\in \calV} \subseteq \calP(\R^d) \times \calP(\R^d)
    \end{equation*}
	  such that for every $v\in \calV$ it holds that $[\mu^{(v)} - \nu^{(v)}]_+$ vanishes on $\supp[\nu^{(v)} - \mu^{(v)}]_+$ and all rectifiable sets of dimension $d-1$, and there exist optimal potentials $(f^{(v)}, g^{(v)})$ for the transport between $[\mu^{(v)} - \nu^{(v)}]_+$ and $[\nu^{(v)} - \mu^{(v)}]_+$ w.r.t.\ $c_1$ and $c_2$ simultaneously,
	then $c_1$ and $c_2$ have the same (unique) OT plan and map for all $v \in \calV$. Moreover, it holds that $(\nabla h_1^*)(y)= (\nabla h_2^*)(y)$ for all
	\begin{equation*}
		y \in \calX \defeq \bigcup_{v \in \calV} \{ \nabla f^{(v)}(x): x \in \opint(\supp [\mu^{(v)} - \nu^{(v)}]_+ )\}\,.
	\end{equation*}
	In particular, let $\calX'$ be open and connected with $l'_1(\calX') \subseteq \{ \norm{x} : x \in \calX\}$, then it holds $l_1 = l_2 + k$ on $\calX'$ for some constant $k \in \R$. Notably, if $0$ is a limit point of $\calX'$ or if there exists $v \in \calV$ such that $\OT_{c_1}(\mu^{(v)}, \nu^{(v)}) = \OT_{c_2}(\mu^{(v)},\nu^{(v)})$ and $$\{ x - y : x \in \supp \mu^{(v)},\, y \in \supp \nu^{(v)}\} \subseteq \calX'\,,$$ we even have $k = 0$.
\end{theorem}
\begin{proof}
	For $v\in \calV$, the OT maps between $\mu^{(v)}_0 \defeq [\mu^{(v)} - \nu^{(v)}]_+$ and $\nu^{(v)}_0 \defeq [\nu^{(v)} - \mu^{(v)}]_+$ are according to \cite[Theorem~6.4]{GaMc} given by
	\begin{equation*}
		T^{(v)}_i(x) = x - \nabla h^*_i(\nabla f^{(v)})\,.
	\end{equation*}
	We define the cost
	\begin{equation*}
		c_{\min}(x,y) = h_{\min}(x-y) \defeq \min(h_1(x-y), h_2(x-y))
	\end{equation*}
	and consider OT between $\mu^{(v)}_0$ and $\nu^{(v)}_0$. The following equality holds:
	\begin{equation*}
		\Phi_{c_{\min}}(\mu^{(v)}_0, \nu^{(v)}_0) = \Phi_{c_1}(\mu^{(v)}_0, \nu^{(v)}_0) \cap \Phi_{c_2}(\mu^{(v)}_0, \nu^{(v)}_0).
	\end{equation*}
	In particular, it holds that $(f^{(v)},g^{(v)})\in \Phi_{c_{\min}}(\mu^{(v)}_0, \nu^{(v)}_0)$. Therefore, it follows for all $x \in \R^d$ that
	\begin{equation*}
		f^{(v)}(x)+ g^{(v)}(T^{(v)}_1(x))\leq h_{\min}(x-T^{(v)}_1(x))\leq h_1(x-T^{(v)}_1(x))\,,
	\end{equation*}
	Using optimality \eqref{eq:dual_map_optimality}, we see that $h_{\min}(x-T^{(v)}_1(x))= h_1(x-T^{(v)}_1(x))$ for $\mu^{(v)}_0$-a.e.\ $x$. Integrating we obtain
	\begin{equation*}
		\int h_{\min}(x-T^{(v)}_1(x)) \de\mu^{(v)}_0(x) = \int h_1(x-T^{(v)}_1(x)) \de\mu^{(v)}_0(x) =\OT_{c_1}(\mu^{(v)}_0,\nu^{(v)}_0),
	\end{equation*}
	but, since $\Phi_{c_{\min}}(\mu^{(v)}_0, \nu^{(v)}_0) \subset \Phi_{c_1}(\mu^{(v)}_0, \nu^{(v)}_0)$,
	\begin{align*}
		\OT_{c_{\min}}(\mu^{(v)}_0,\nu^{(v)}_0) &= \sup_{(f,g)\in \Phi_{c_{\min}}(\mu^{(v)}_0, \nu^{(v)}_0)} \int f \de\mu^{(v)}_0 +\int g \de\nu^{(v)}_0 \\
		&\leq \sup_{(f,g)\in \Phi_{c_1}(\mu^{(v)}_0, \nu^{(v)}_0)} \int f \de\mu^{(v)}_0 +\int g \de\nu^{(v)}_0 \\
		&= \int f^{(v)} \de\mu^{(v)}_0 + \int g^{(v)} \de\nu^{(v)}_0 = \OT_{c_1}(\mu^{(v)}_0,\nu^{(v)}_0)\,.
	\end{align*}
	Therefore, $\OT_{c_{\min}}(\mu^{(v)}_0,\nu^{(v)}_0) = \OT_{c_1}(\mu^{(v)}_0,\nu^{(v)}_0)$ and $\pi_{\min}=(I\times T^{(v)}_1)_{\#} \mu^{(v)}_0 $ is a OT plan w.r.t.\ $c_{\min}$. We now see that $\pi_{\min}$ is unique. To this end, note that $c_{\min}(x,y) = l_{\min}(\norm{x-y})$ where $l_{\min} \defeq \min[l_1, l_2]$ is strictly concave and continuous with $l_{\min}(0) = 0$. Hence, uniqueness follows from \cite[Theorem~6.4]{GaMc}.

	The choice of $i=1$ was arbitrary so that $\pi_{\min}=(I\times T^{(v)}_2)_{\#} \mu^{(v)}_0$ holds as well. This proves that for $\mu_0^{(v)}$-a.e.\ $x$ it holds $T^{(v)}_2(x) =T^{(v)}_1(x)$ and therefore
	\begin{equation*}
		(\nabla h_1^*)[\nabla f^{(v)}(x)]= (\nabla h_2^*)[\nabla f^{(v)}(x)]\,.
	\end{equation*}
	Hence, we have that $ (\nabla h_1^*)(y)= (\nabla h_2^*)(y)$ for all $y \in \calX$. In particular, this implies that $h_1^* = h_2^* + k$ on each open and connected component of $\calX$ with $k \in \R$. If these components cover $[0, \infty)$ in norm, we have by definition of $h_1^*$ and $h_2^*$ that
	\begin{equation*}
		(-l_1)^*(-y) = (-l_2)^*(-y) + k\,, \quad y \in [0, \infty)\,,
	\end{equation*}
	and $l_1 = l_2$ follows from the uniqueness of the convex conjugate (Fenchel–Moreau theorem) for $-l_1$ and $-l_2$ and the assumption that $l_1(0) = 0 = l_2(0)$.

	If $l_1$ and $l_2$ are differentiable, we get for $x \in \R^d \setminus \{ 0 \}$ and $i \in \{1, 2\}$ that
	\begin{equation*}
		\nabla h_i^*(x) = [(-l_i)^*]'(-\norm{x}) \frac{x}{\norm{x}} = (l'_i)^{-1}(\norm{x})\frac{x}{\norm{x}}\,.
	\end{equation*}
	This implies for all $y \in \R^d \setminus \{0\}$ that
	\begin{equation*}
		(l'_1)^{-1}(\norm{y}) \frac{y}{\norm{y}} = (l'_2)^{-1}(\norm{y}) \frac{y}{\norm{y}} \implies (l'_1)^{-1}(\norm{y}) = (l'_2)^{-1}(\norm{y})\,.
	\end{equation*}
	Given some open and connected subset $\calX' \subseteq \R$ with $l'_1(\calX') \subseteq \{ \norm{x} : x \in \calX \}$, it follows for $z \in \calX'$ that
	\begin{equation*}
		z = (l'_1)^{-1}(l'_1(z)) = (l'_2)^{-1}(l'_1(z))\,,
	\end{equation*}
	and therefore $l'_2 = [(l'_2)^{-1}]^{-1} = l'_1$ on $\calX'$. Hence, we conclude that $l_1 = l_2 + k$ on $\calX'$ for some constant $k \in \R$. If $0$ is a limit point of $\calX'$, we must have $0 = l_1(0) = l_2(0) + k = k$. If $\{ \norm{x-y} : x \in \supp \mu^{(v)}, \, y \in \supp \nu^{(v)} \}\subseteq \calX'$ for one $v \in \calV$, then it follows that $c_1 = c_2 +k$ on $\supp\mu^{(v)} \times \supp\nu^{(v)}$. This together with the assumption that $\OT_{c_1}(\mu^{(v)}, \nu^{(v)}) = \OT_{c_2}(\mu^{(v)}, \nu^{(v)})$ implies $k=0$.
\end{proof}

Note that \autoref{thm:concave_cost_dual_sol_known} only requires knowledge of the optimal potentials and not of the OT plans. This is due to the concavity of the cost function and, in contrast, is not the case for convex costs, see \autoref{thm:convex_cost_pi_dual_sol_known}.

To get rid of the assumption of knowing the optimal potentials, we proceed as before for convex costs and use \autoref{thm:same_dual_sol}. To this end, we provide uniqueness conditions for the optimal potentials for concave costs. Following the proof for convex costs \cite[Collary~2.7]{del2021central}, which relies on the Lipschitzianity of potentials, we see that this is more involved for concave costs as the derivative $l'$ may have infinite slope at $0$. We evade this by constraining the support of $\mu$ and $\nu$.

\begin{lemma} \label{lemma:concave_cost_unique_potentials_supp}
	Let $c : \R^d \times \R^d \to [0, \infty)$ be induced by $l \circ \norm{\argdot}$ with $l \in \calL$. Then, it holds for probability measures $\mu,\,\nu \in \calP(\R^d)$ such that $\mu \ll \lambda_d$ has connected support and negligible boundary as well as $\supp\mu \cap \supp\nu = \emptyset$ that their optimal potentials w.r.t.\ $c$ are unique (up to additive constants and on $\opint [\supp \mu]$).
\end{lemma}
\begin{proof}
	As $\supp\mu \cap \supp\nu = \emptyset$, note that
	\begin{equation*}
		\delta \defeq \inf \{ \norm{x - y} : x \in \supp\mu,\, y \in \supp\nu \} > 0\,.
	\end{equation*}
	Together with the fact that $l'$ is decreasing, it follows that $l$ is $l'(\delta)$-Lipschitz continuous on $\{ \norm{x - y} : x \in \supp\mu,\, y \in \supp\nu \}$. As a consequence, optimal potentials $f_1$ and $f_2$ are also Lipschitz continuous.

	Further, by uniqueness of the OT map \cite[Theorem~6.4]{GaMc}, it follows $\mu$-a.e.\ that
	\begin{equation*}
		y \defeq \nabla h^*(\nabla f_1(x)) = \nabla h^*(\nabla f_2(x))\,,
	\end{equation*}
    where we recall that $h^*$ is the concave conjugate of $h$. As
	\begin{equation*}
		\nabla h(z) = \frac{l'(\norm{z})}{\norm{z}} z
	\end{equation*}
	is invertible, note that $\nabla h^* = [\nabla h]^{-1}$. Hence, we get for $\mu$-a.e.\ $x$ that
	\begin{equation*}
		\nabla h(y) = \nabla f_1(x) = \nabla f_2(x)\,,
	\end{equation*}
	which implies uniqueness as in the proof of \cite[Collary~2.7]{del2021central}.
\end{proof}

Using \autoref{lemma:concave_cost_unique_potentials_supp} to get uniqueness, we can change the assumption of knowing the optimal potentials to knowing the values and marginals on an open set. This way, we obtain a result were we only need the marginals and OT costs between them.

\begin{corollary}
\label{thm:concave_cost_dual_sol_unknown} Let $c_i : \R^d \times \R^d \to [0, \infty)$ be induced by $h_i \defeq l_i \circ \norm{\argdot}\in \calL_{p, \norm{\argdot}}$, i.e., $c_i(x,y) = h_i(x-y)$, for $i \in \{1, 2\}$. \\
Assume that there exists a dense subset $\{(\mu^{(v)},\nu^{(v)})\}_{v\in \calV}$ of a non-empty open set of $\calP_p(\R^d) \times \calP_p(\R^d)$ satisfying
		\begin{equation*}
			\OT_{c_1}(\mu^{(v)}, \nu^{(v)}) = \OT_{c_2}(\mu^{(v)},\nu^{(v)}) \qquad \text{for all } v\in \calV\,.
		\end{equation*}
	Define the set
	\begin{equation*}
		\calV_\emptyset \defeq \left\{ v \in \calV \,\middle|\,
		\begin{matrix}
			\supp \mu^{(v)} \cap \supp \nu^{(v)} = \emptyset \text{ and } \\
			\text{$\mu^{(v)} \ll \lambda_d$ has connected support and negligible boundary}
		\end{matrix}
		\right\}\,.
	\end{equation*}
	Then, it holds that $c_1$ and $c_2$ have the same (unique) OT plan, map and optimal potential $f^{(v)}$ for all $v \in \calV_\emptyset$. Moreover, $(\nabla h_1^*)(y)= (\nabla h_2^*)(y)$, for all
	\begin{equation*}
		y\in \calX \defeq \bigcup_{v \in \calV_\emptyset} \{ \nabla f^{(v)}(x): x \in \opint(\supp \mu^{(v)})\}\,.
	\end{equation*}
	In particular, let $\calX'$ be open and connected with $l'_1(\calX') \subseteq \{ \norm{x} : x \in \calX \}$, then it holds $l_1 = l_2 + k$ on $\calX'$ for some constant $k \in \R$. Notably, if $0$ is a limit point of $\calX'$ or if there exists $v \in \calV$ such that $$\{ x - y : x \in \supp \mu^{(v)},\, y \in \supp \nu^{(v)}\} \subseteq \calX'\,,$$ we even have $k = 0$.
\end{corollary}
\begin{proof}
	For $v \in \calV_\emptyset$ we have by $\supp \mu^{(v)} \cap \supp \nu^{(v)} = \emptyset$ that $[\mu^{(v)} - \nu^{(v)}]_+ = \mu^{(v)}$ and $[\nu^{(v)} - \mu^{(v)}]_+ = \nu^{(v)}$. Furthermore, by \autoref{thm:same_dual_sol} in conjunction with \autoref{lemma:concave_cost_unique_potentials_supp} there exists a unique (up to additive constants) optimal potential $f^{(v)}$ for $c_1$ and $c_2$. Hence, the assertion follows from \autoref{thm:concave_cost_dual_sol_known}.
\end{proof}

\section{Convex Cost Functions on the Real Line} \label{sec:real_line}

We already discussed identifiability of convex cost functions on $\R^d$ in \autoref{subsec:convex_cost}. Our approach used the explicit formula of the OT map by Gangbo and McCann \cite{GaMc}. Notably, in the case of $d = 1$ the OT cost additionally admits a closed-form expression in terms of the marginals and the inducing convex function. This will allow us to derive identifiability results when we only observe the marginals and the OT costs.

Consider the real line $\R$ that is equipped with a cost function of the form $c(x, y) = h(x - y)$ where $h : \R \to [0, \infty)$ is convex with $h(0) = 0$, i.e., $c$ is induced by $h$. Denote with $\calH^1$ the set containing all such functions. According to \cite[Proposition~2.17]{Santambrogio2015} it holds that the OT cost between the probability measures $\mu^{(v)},\,\nu^{(v)} \in \calP(\R)$ with cumulative distribution functions $F_v$ and $G_v$, respectively, is given by
\begin{equation} \label{eq:def_a_h}
	\alpha_h(F_v, G_v) \defeq \int_0^1 h(F_v^{-1}[u] - G_v^{-1}[u]) \de{u} = \int_{-\infty}^\infty \frac{1}{\abs{[\Gamma_v' \circ \Gamma_v^{-1}](x)}} h(x) \de{x}\,,
\end{equation}
where $F_v^{-1}$, $G_v^{-1}$ are the quantile functions and $\Gamma_v(u) \defeq F_v^{-1}[u] - G_v^{-1}[u]$. As $h \geq 0$, note that this integral is always defined, and we assume that it is finite. Here, we assume that there exists a sequence $\{A_i\}_{i \in I}$ of disjoint open intervals on $(0,1)$ with $\lambda_1(\bigcup_{i \in I} A_i) = 1$ such that either $\Gamma_{v}' > 0$ or $\Gamma_{v}' < 0$ on $A_i$ for all $i \in I$. Then, we understand the r.h.s.\ of \eqref{eq:def_a_h} as
\begin{equation*}
	\frac{1}{\abs{[\Gamma_v' \circ \Gamma_v^{-1}](x)}} = \sum_{i \in I} \frac{1}{\abs{ [\Gamma_v' \circ (\restr{\Gamma_v}{A_i})^{-1}] (x)}} \indicfunc{x \in \Gamma_v(A_i)}\,.
\end{equation*}
In particular, we exclude that $F_v^{-1} = G_v^{-1}$ on a set of positive Lebesgue measure, but we allow for possibly infinite intersection points of $F_v^{-1}$ and $G_v^{-1}$.

As it is often too restrictive to assume that $h$ is integrable (e.g. if $h(x) = \abs{x}^p$, $p > 0$), we will reweight \eqref{eq:def_a_h} w.r.t.\ some $F_v$, say $F_0$. For example, in order to define a proper Hilbert space via \eqref{eq:def_a_h} let us introduce an absolutely continuous base measure with c.d.f. $F_0$ associated to a family of marginal pairs $\{(F_v, G_v)\}_{v\in\calV}$. Then,
\begin{equation*}
	\alpha_h(F_v, G_v) = \int_{-\infty}^\infty \frac{1}{f_0(x) \abs{[\Gamma_v' \circ \Gamma_v^{-1}](x)}} h(x) \de{F_0(x)}\,,
\end{equation*}
where $f_0 = F_0'$ is the Lebesgue density of $F_0$.

\begin{theorem}
	Let $F_0$ be some absolutely continuous base measure with Lebesgue density $f_0$ that is associated to a family of marginal pairs $\{(F_v, G_v)\}_{v\in\calV}$ such that $\{ 1 / (f_0 \abs{\Gamma_v' \circ \Gamma_v^{-1}}) \}_{v \in \calV}$ is a dense family of functions in $L^2(F_0)$. Then, for all $h \in \calH^1 \cap L^2(F_0)$ the pairings $\{ \alpha_h(F_v, G_v), F_v, G_v \}_{v\in\calV}$ determine $h$ $\lambda_1$-a.e.
\end{theorem}
\begin{proof}
	As $L^2(F_0)$ is a Hilbert space, the inner product on a dense subset uniquely determines the function $h$.
\end{proof}

\subsection{Location-scale Families} \label{subsubsec:loc_scal_fam}

Unfortunately, $\Gamma_v' \circ \Gamma_v^{-1}$ is not easily computed and the conditions of the last theorem are not easily verified, in general. An exceptional case is given by location-scale families. We assume that all c.d.f.s. are of the form
\begin{equation*}
	G_{a,b}(x) = G \left(\frac{x-a}{b}\right)\,, \qquad (a,b) \in R\,,
\end{equation*}
where $G = G_{0,1}$ is the generator of the location-scale family and $R \defeq \R \times (0, \infty)$. $G$~then serves as the base measure in the last theorem. This relationship implies that
\begin{equation*}
	G_{a,b}^{-1}[u] = a + b\, G^{-1}[u]\,.
\end{equation*}
Hence, \eqref{eq:def_a_h} becomes
\begin{equation*}
	\alpha_h(G_{a,b}, G) = \int_0^1 h(a + (b - 1) G^{-1}[u]) \de{u}\,, \qquad (a, b) \in R\,.
\end{equation*}
If $G$ is strictly increasing and differentiable, we have that $g \defeq G' > 0$ and for $b \neq 1$ that $\Gamma_{a,b}$ is invertible and differentiable everywhere with
\begin{equation*}
	\Gamma_{a,b}^{-1}[x] = G \left(\frac{x - a}{b - 1}\right)\,, \qquad \Gamma_{a,b}'(u) = \frac{b - 1}{[g \circ G^{-1}](u)}\,.
\end{equation*}
Thus, \eqref{eq:def_a_h} becomes
\begin{align*}
	\alpha_h(G_{a,b}, G) &= \int_{-\infty}^\infty (b-1) \abs*{ \frac{b-1}{[G' \circ G^{-1} \circ G]([x - a] / [b - 1])}}^{-1} \de{x} \\
	&= \int_{-\infty}^\infty g\left(\frac{x - a}{b - 1} \right) h(x) \de{x}\,.
\end{align*}
For two functions $h_1,\, h_2 \in \calH^1$, assume now
\begin{equation*}
	\alpha_{h_1}(G_{a,b},G) = \alpha_{h_2}(G_{a,b}, G) \qquad \text{for all } (a,b)\in R\,.
\end{equation*}
Reparametrizing and setting $h \defeq h_1 - h_2$, this can be written as
\begin{equation} \label{eq:g_transform_def}
	\calI_g[h](a,b) \defeq \int_{-\infty}^\infty g\left(\frac{x-a}{b} \right) h(x) \de{x} = 0\,.
\end{equation}
We call this the $g$-transform of $h$, which induces a linear integral operator on $L^1(G)$. Now it becomes apparent that uniqueness of $h$ is intimately related to injectivity of linear Fredholm operators of first kind (see \cite{Kress2014}) as well as to statistical completeness, see e.g. \cite[Section~1.5]{Pfanzagl1994}. Here, a family of probability measures $\calP$ on a joint measurable space $(\calX, \mathcal{A})$ is denoted as complete, if for every $\calP$-integrable function $f : (\calX, \mathcal{A}) \to \R$, it holds that $\int f \de{P} = 0$ for all $P \in \calP$ implies that $f \equiv 0$ $P$-a.e.\ for all $P \in \calP$.

\begin{theorem} \label{thm:convex_cost_onedim_g_transform}
	Assume that the cost $c$ is induced by a function $h \in \calH' \subseteq \calH^1$. Furthermore, suppose that the marginal distributions are related by a location-scale family with generator $G$ that is strictly increasing and differentiable with Lebesgue density $g$. Then, the cost $c$ is identifiable in the class of cost functions induced by $\calH' \cap L^1(G)$ if and only if the $g$-transform is injective on $\calH' \cap L^1(G)$.
\end{theorem}

\begin{remark}
	Let $\calH' \subseteq \calH^1$ be a subset of continuous functions that are ordered in the sense that for all $h_1,\,h_2 \in \calH'$ it holds that either $h_1 \leq h_2$ or $h_1 \geq h_2$ $\lambda_1$-a.e. Then, it follows from \eqref{eq:g_transform_def} that $c$ is identifiable.
\end{remark}

Note that \cite{Asanov2019} obtain injectivity of Fredholm operators of the first kind under certain conditions on the integration kernel. In our setting, this translates to various (complicated) conditions on $g$ and its derivatives.

\subsubsection*{Location Families}

Notably, for fixed $b$ (w.l.o.g.\ $b = 1$) the $g$-transform is equal to the convolution of $g$ and $h$,
\begin{equation*}
		(g * h)(a) \defeq \int_{-\infty}^\infty g(x-a)h(x) \de{x} = \calI_g[h](a,1)\,.
\end{equation*}
Denoting with $\calF : L^1(\lambda_1) \to \calC(\R)$ the Fourier transform (operator), we get that $\calF(g * h) = \calF(g) \cdot \calF(h)$. In particular, if $\calF(g) \neq 0$ it follows that $\calF(h) = 0$ and thus $h = 0$. However, for this to hold we require that $h \in L^1(\lambda_1)$. This is a very strong condition on $h = h_1 - h_2$ that severely limits the class $\calH' \subseteq \calH^1$.

\begin{example} \label{ex:loc_scale_gaussian_one_dim}
	Let the location-scale family be the normal distributions, we then obtain the underlying density
	\begin{equation*}
		g(x) = \frac{1}{\sqrt{2\pi}} \exp\left( -\frac{1}{2}x^2 \right)\,,
	\end{equation*}
	and
	\begin{equation*}
		\calI_g[h](a,b) = \frac{1}{\sqrt{2 \pi}} \int_{-\infty}^\infty \exp\left(-\frac{1}{2} \frac{(x-a)^2}{b^2}\right) h(x) \de{x} \,.
	\end{equation*}
	Notably, for $b = \sqrt{2}$ this is equal (up to scaling) to the Weierstrass transform, see e.g.\ \cite{Bilodeau1961}, which is known to be injective on $L^1(G)$, and similarly for any $b > 0$. Hence, identifiability holds true for the class of $G$-integrable cost functions.
\end{example}

\begin{example}
	Let the location-scale family be the Cauchy distributions, i.e., we have the underlying density
	\begin{equation*}
		g(x) = \frac{1}{\pi} \frac{1}{1 + x^2}\,,
	\end{equation*}
	and
	\begin{equation*}
		\calI_g[h](a,b) = \frac{1}{\pi} \int_{-\infty}^\infty \frac{1}{1 + \left( \frac{x-a}{b} \right)^2} h(x) \de{x}\,.
	\end{equation*}
	For $b = 1$, this is also called the Poisson transform \cite{Pollard1955}. Here, injectivity holds for all $h \in \calH^1 \in L^1(G)$ such that the function $t \mapsto \int_{t_0}^t h(x) \de{x}$ is of bounded variation for all $t_0 \in \R$. This holds true if $h$ has finite total variation on $\R$, i.e., if $\int_{-\infty}^\infty \abs{h'(x)} \de{x} < \infty$. Injectivity of the Poisson transform has also been shown by \cite{Mattner1992} where also other location families with injective $g$-transforms like the log-gamma distribution or Student's t-distribution are discussed to which \autoref{thm:convex_cost_onedim_g_transform} applies as well.
\end{example}

\begin{remark}
	If we restrict $\calH' \subseteq \calH^1$ to only include bounded functions (in statistical terms often called boundedly complete \cite[Section~1.5]{Pfanzagl1994}), injectivity of the $g$-transform for a location family can be fully characterized with the Fourier transform of the probability measure with density $g$ \cite{Mattner1993}. Namely, the $g$-transform is injective if and only if the Fourier transform has no zeros.
\end{remark}

\subsubsection*{Scale Families}

\begin{example}
	Let the location-scale family be the Laplace distributions, i.e., it is generated by the density
	\begin{equation*}
		g(x) = \frac{1}{2} \exp(-\abs{x})\,,
	\end{equation*}
	and
	\begin{equation*}
		\calI_g[h](a,b) = \frac{1}{2} \int_{-\infty}^\infty \exp\left(-\frac{\abs{x-a}}{b}\right) h(x) \de{x}\,.
	\end{equation*}
	For $a = 0$ and symmetric $h$, this reduces to the (one-sided) Laplace transform
	\begin{equation*}
		\calI_g[h](0, 1/b) = \int_{0}^\infty \exp(-x b) h(x) \de{x}\,.
	\end{equation*}
	Note that this is the $g$-transform for the scale family of exponential distributions with underlying density $x \mapsto \exp(-x) \indicfunc{x \geq 0}$.

	Furthermore, if $\calH' \subseteq \calH^1$ is the subset of continuous and symmetric functions that are of any exponential order, i.e., for all $h \in \calH'$ and $b > 0$ there exists a constant $M > 0$ such that $h(x) \leq M \exp(xb)$ for all $x \geq 0$, injectivity of the $g$-transform in $\calH'$ follows from Post's inversion formula \cite{Post1930}.
\end{example}

\subsection{Radially Symmetric Cost Functions}

Unfortunately, there is no analogue for \eqref{eq:def_a_h} in $\R^d$, $d > 1$, see also \autoref{subsec:convex_cost}. However, for costs $c : \R^d \times \R^d \to [0, \infty)$ that are induced by the Euclidean norm $\norm{\argdot}$, that is $c(x,y) = h(\norm{x-y})$ for some $h : [0, \infty) \to [0, \infty)$, there are certain families of marginals such that the $d$-dimensional OT problem reduces to a one-dimensional one. For instance, given a unit vector $u \in \R^d$ and base $r \in \R^d$, denote the affine transformation $A^{u,r} : \R \to \R^d$, $x \mapsto x u + r$. Then, it holds for two probability measures $\mu$ and $\nu$ on $\R$ that
\begin{equation*}
	\OT_c(A^{u,r}_\#\mu, A^{u,r}_\#\nu) =\OT_{\tilde{c}}(\mu, \nu)\,,
\end{equation*}
where $\tilde{c}$ is induced by $h \circ \abs{\argdot}$. Hence, \autoref{thm:convex_cost_onedim_g_transform} can be extended to cost functions that are induced by $\calH^1_{\norm{\argdot}} = \{ h \circ \norm{\argdot} : h \in \calH^1\}$.

\begin{corollary} \label{cor:convex_cost_d_dim_g_transform}
	Assume that the cost function $c$ is induced by $\calH'_{\norm{\argdot}}$ with $\calH' \subseteq \calH^1$. Furthermore, suppose that the marginal distributions are affine pushforwards of a location-scale family with generator $G$ that is strictly increasing and differentiable with Lebesgue density $g$. Then, the cost $c$ is identifiable in the class of cost functions induced by $[\calH' \cap L^1(G)]_{\norm{\argdot}}$ if and only if the $g$-transform is injective on $\calH'_{\abs{\argdot}} \cap L^1(G)$.
\end{corollary}

\begin{remark}
	Note that the one-dimensional affine subspace where the probability measures are projected to in \autoref{cor:convex_cost_d_dim_g_transform} can be different for each observed pair of marginals.
\end{remark}

\begin{example}
	Let the underlying one-dimensional location-scale family be the normal distributions, recall also \autoref{ex:loc_scale_gaussian_one_dim}. In this case, for a one-dimensional normal distribution $\calN_1(a,b^2)$ with mean $a \in \R$ and variance $b^2 > 0$ it holds that
	\begin{equation*}
		A^{u,r}_\# \calN_1(a,b^2) = \calN_d(a u + r, b^2 u u^\transp)\,.
	\end{equation*}
	In particular, this implies that we have identifiability of cost functions induced by $[\calH^1_{\abs{\argdot}} \cap L^1(G)]_{\norm{\argdot}}$ if we observe the OT costs between multivariate normal distributions with covariance matrices that have rank $1$.
\end{example}

\section{Conclusion and Discussion}

In this work, identifiability of the cost function in the inverse OT problem has been conducted for continuous probability distributions. We investigated the class of convex and concave cost functions and gave identifiability criteria depending on certain combinations of information on the plans, the potentials and the optimal values. We discuss some extensions and open problems which are not addressed in this work:

\begin{enumerate}
	\item More general cost functions and ground spaces. The case where distributions are finitely supported or discrete. In this case, the identifiability conditions when the plans and total costs are known are obtained as a linear system of equations and can be dealt with by elementary tools from linear algebra. The case in which the OT plans are unknown is more complicated since, in general, the uniqueness of OT potentials does not hold. Hence, we cannot apply the strategy used in this work to recover the OT potentials from the OT cost values.
	\item The (entropy) regularized case. Here, similar results to those obtained in this work can be conjectured because, in general, regularized transport potentials are usually unique. Furthermore, once the values of the potentials are obtained, and knowing the values of the marginal distributions, the optimality conditions (see \cite{nutz2021introduction}) allow us to transform the inverse OT problem into a standard linear inverse problem.
    \item The unbalanced case. In this work, we always assume that the observed marginal measures $\{ (\mu^{(v)}, \nu^{(v)} )\}_{v \in \calV}$ are probability measures and therefore have the same total mass. When this is not the case, we have an unbalanced OT problem for which there are various formulations, see e.g. \cite{Benamou2003numerical,Chizat2018,heinemann2023kantorovich}. Due to the lack of optimal maps, we believe that our approach cannot be extended to deal with identifiability of the cost function for unbalanced OT.
    \item We give identifiability criteria for the underlying cost function $c$, but do not touch upon its estimation. If the cost function $c$ is identifiable, the question arises if, e.g., the procedure of \cite{InverseSIam} converges to the unique underlying cost function. This will be postponed to future investigations.
	\item Finally, the problem of the existence of a solution of the inverse OT also remains untreated. In future work, we will investigate the necessary and/or sufficient conditions for the inverse OT problem to have at least one solution. That is, the existence of a cost $c$ such that
	\begin{equation} \label{eq:existence}
	\alpha^{(v)} = \OT_c(\mu^{(v)}, \nu^{(v)}) = \int c \de{\pi^{(v)}} = \int f^{(v)} \de{\mu^{(v)}} + \int g^{(v)} \de{\nu^{(v)}}\,,
\end{equation}
	for a given set $\{ (\mu^{(v)}, \nu^{(v)}) \}_{v \in \calV} \subseteq \calP(\R^d) \times \calP(\R^d)$. Note that this is a very different problem from identifiability. Namely, it needs to be shown that a cost function $c$ with \eqref{eq:existence} exists and that the observed OT information is optimal for this $c$. For the latter, we need to consider all possible plans (or potentials) and show that they cannot be better in terms of the OT cost than the observations for this $c$.
\end{enumerate}

\printbibliography[title = References]

\end{document}